\documentclass[oneside,10pt]
{amsart}
\usepackage[b5paper]{geometry}	    
\usepackage{amsfonts,amsmath,latexsym,amssymb} 
\usepackage{amsthm}                
\usepackage{mathrsfs,upref}         
\usepackage{mathptmx}		    

\usepackage[active]{srcltx}

\usepackage{graphicx}

\usepackage{enumerate}

\usepackage[final]
{showkeys}
\usepackage{url}

\RequirePackage[OT1]{fontenc}

\usepackage{tikz}
\usetikzlibrary{decorations}
\usetikzlibrary{decorations.pathreplacing}

\usepackage{hyperref}

\newtheorem{theorem}{Theorem}[section]           
               
\newtheorem{corollary}[theorem]{Corollary}
\newtheorem{proposition}[theorem]{Proposition}

\theoremstyle{definition}

\newtheorem{remark}[theorem]{Remark}

\numberwithin{equation}{section}       

\renewcommand{\gg}{>\kern-2pt>}
\renewcommand{\ll}{<\kern-2pt<}

\renewcommand{\gg}{>\kern-2pt>}
\renewcommand{\ll}{<\kern-2pt<}

\renewcommand{\le}{\leqslant}
\renewcommand{\ge}{\geqslant}

\newcommand{\Ga}{\Gamma}

\newcommand{\de}{\delta}

\newcommand{\LD}{\mathcal{L}\!\mathcal{D}}
\renewcommand{\LD}{\mathcal{L}{\kern -1.9pt}\mathcal{D}}
\renewcommand{\LD}{\mathcal{D}}
\renewcommand{\LD}{\mathcal{L}{\kern -4pt}\mathcal{C}}
\renewcommand{\LD}{\mathcal{R}{\kern -3pt}\mathcal{C}}

\newcommand{\widesim}[2][1.5]{
  \mathrel{\underset{#2}{\scalebox{#1}[1]{$\sim$}}}
}

\begin{document}

\title[Identities and inequalities for the cosine and sine functions]{Identities and inequalities for the cosine and sine functions}


\author{Iosif Pinelis}

\address{Iosif Pinelis, Department of Mathematical Sciences\\
Michigan Technological University\\
Hough\-ton, Michigan 49931, USA
}
\email{ipinelis@mtu.edu}



\date{07.10.2018
}                               





\begin{abstract}
        Identities and inequalities for the cosine and sine functions are obtained. 
\end{abstract}

\maketitle



\section{Statements and discussion}\label{results}

The basic result of this note is 

\begin{theorem}\label{th:} 
For any real $x$ 
\begin{equation}\label{eq:expan}
	\cos\pi x=\sum_{j=1}^\infty t_j \pi^{2j} (1/4-x^2)^j,
\end{equation}
where 
\begin{equation}\label{eq:t,a}
	t_j:=\sum_{k=0}^\infty a_{j,k},\quad 
	a_{j,k}:=\frac{(-\pi^2/4)^k}{(2j+2k)!}\,\binom{j+k}j. 
\end{equation}

Moreover, one has the recurrence 
\begin{gather}\label{eq:recur}
	t_j=\frac{2(2j-3)}{\pi^2 j}\,t_{j-1}-\frac1{\pi^2 j(j-1)}\,t_{j-2}\quad\text{for }j=2,3,\dots, 
\end{gather}
with $t_0=0$ and $t_1=1/\pi$. 

Furthermore, for all natural $j$  
\begin{equation}\label{eq:<t_j<}
	0<t_j<\frac1{(2j)!},  
\end{equation}
and 
\begin{equation*}
	t_j\sim\frac1{(2j)!}\quad\text{as }j\to\infty. 
\end{equation*}
\end{theorem}

The necessary proofs will be given in Section~\ref{proofs}. 

Recurrence \eqref{eq:recur} allows one to compute the coefficients $t_j$ in \eqref{eq:expan} quickly and efficiently. In particular, 
we see that  
\begin{align*}
	(t_1,\dots,t_5)&=\left(\frac{1}{\pi },\frac{1}{\pi ^3},\frac{12-\pi ^2}{6 \pi ^5},\frac{10-\pi ^2}{2 \pi
   ^7},\frac{1680-180 \pi ^2+\pi ^4}{120 \pi ^9}\right) \\ 
   &\approx\big(0.318,0.0323,1.16\times10^{-3},2.16\times10^{-5},2.46\times10^{-7}\big). 
\end{align*}

On the other hand, inequalities \eqref{eq:<t_j<} together with identity \eqref{eq:expan} will serve as the source of other inequalities, which begin with the following:  

\begin{corollary}\label{cor:}
For each natural $m$, consider the polynomial 
\begin{equation}\label{eq:P_m}
P_m(x):=\sum_{j=1}^m t_j \pi^{2j} (1/4-x^2)^j,	
\end{equation}
which is the $m$th partial sum of the series in \eqref{eq:expan}. 
Then for all $x\in(-1/2,1/2)$ 
\begin{equation}\label{eq:mono}
	P_m(x)<P_{m+1}(x)<\cos\pi x
\end{equation}
and 
\begin{equation}\label{eq:bound}
\begin{aligned}
	0<\de_m(x):=\cos\pi x-P_m(x)
	&<\frac{\pi^{2m+2}(1/4-x^2)^{m+1}}{(2m+2)!}\,\frac1{1-q_m} \\ 
	&\widesim[2.5]{m\to\infty}\de_m^*(x):=\frac{\pi^{2m+2}(1/4-x^2)^{m+1}}{(2m+2)!},  
\end{aligned}    
\end{equation}
where 
\begin{equation*}
	q_m:=\frac{\pi^2/4}{(2m+4)(2m+3)}  
\end{equation*}
and the asymptotic relation holds uniformly in $x\in(-\frac12,\frac12)$. 
\end{corollary}

\begin{remark}\label{rem:interp}
Note that the function $\de_m$ is analytic. So, 
in view of \cite[Proposition~I]{karlin-karon} (proved e.g.\ in \cite[page~29]{davis}), it follows from \eqref{eq:bound} that, for each natural $m$, $P_m(x)$ is the Hermite interpolating polynomial (HIP) (of degree $2m$) determined by the $2m+2$ conditions $\de_m^{(j)}(\pm\frac12)=0$ for $j=0,\dots,m$; 
in fact, by P\'olya's theorem \cite[Theorem~I]{karlin-karon}, the polynomial $P_m(x)$ is already determined by any $2m+1$ of the just mentioned $2m+2$ conditions. 

Explicit expressions of the general HIP were presented, in particular, in \cite{greville,spitzbart}. It is unclear, though, how to use those results to show 
that the polynomial $P_m(x)$, as defined in \eqref{eq:P_m}, is the HIP; nor is it seen how to derive 
monotonicity property \eqref{eq:mono} or the bound in \eqref{eq:bound} from the mentioned expressions. 

Least-squares approximating polynomials of the form 
\begin{equation}\label{eq:milo}
\sum_{j=1}^m d_{m,j}(1-x^2)^j	
\end{equation}
for 
$\cos(\pi x/2)$ and $x\in[-1,1]$ were given in \cite[Example~4.1]{milo86}; see also \cite{milo87} and \cite[\S5.3.2]{mastro-milo}. In contrast with the coefficients $t_j \pi^{2j}$ 
in \eqref{eq:P_m}, the coefficients $d_{m,j}$ 
in \eqref{eq:milo} depend on $m$. 
%
%
\qed
\end{remark}

One also has 
\begin{proposition}\label{prop:}
For all natural $j$ 
\begin{equation}\label{eq:t=}
		t_j=\frac{\pi^{1-j}}{2j!}\, J_{j-1/2}(\pi/2),   
\end{equation}
where 
\begin{equation}\label{eq:J}
	J_\nu(z):=\sum_{k=0}^\infty \frac{(-1)^k (z/2)^{\nu+2k}}{k!\,\Ga(\nu+k+1)}
\end{equation}
is an expression defining the Bessel function (of the first kind) -- as e.g.\ is done in \cite[page~359]{whittaker-watson}. 
\end{proposition}

In view of the identity $\sin\pi x=\cos\pi(x-1/2)$, one immediately obtains the corresponding results for $\sin\pi x$ instead of $\cos\pi x$. More specifically, we have 

\begin{corollary}\label{cor:sin}
Take any real $x$. Then  	
\begin{equation}\label{eq:expan-sin}
	\sin\pi x=\sum_{j=1}^\infty t_j \pi^{2j} (x(1-x))^j. 
\end{equation}

Also, for all natural $m$ and all $x\in(0,1)$ 
\begin{equation}\label{eq:mono-sin}
	Q_m(x)<Q_{m+1}(x)<\sin\pi x
\end{equation}
and 
\begin{equation}\label{eq:bound-sin}
\begin{aligned}
	0<\sin\pi x-Q_m(x)
	&<\frac{\pi^{2m+2}(x(1-x))^{m+1}}{(2m+2)!}\,\frac1{1-q_m} \\ 
	&\widesim[2.5]{m\to\infty}\frac{\pi^{2m+2}(x(1-x))^{m+1}}{(2m+2)!},  
\end{aligned}    
\end{equation} 
where
\begin{equation}\label{eq:Q_m}
Q_m(x):=P_m(x-1/2)=\sum_{j=1}^m t_j \pi^{2j} (x(1-x))^j. 	
\end{equation}
\end{corollary}


\begin{remark}\label{rem:compar}
One may compare expansion \eqref{eq:expan-sin} with the Maclaurin expansion 
\begin{equation}\label{eq:mac-sin}
	\sin\pi x=-\sum_{j=1}^\infty \frac{(-\pi x)^{2j-1}}{(2j-1)!}. 
\end{equation}
For any natural $m$, the approximation of $\sin\pi x$ by the corresponding Maclaurin polynomial 
\begin{equation*}
	S_m(x):=-\sum_{j=1}^m \frac{(-\pi x)^{2j-1}}{(2j-1)!}
\end{equation*}
is exact to order $2m$ at $x=0$, but it is not exact to any order at $x=1$. 
In contrast, in view of \eqref{eq:mono-sin}, the approximation of $\sin\pi x$ by $Q_m(x)$ 
is exact to order $m$ at both $x=0$ and $x=1$. 
Also, in view of \eqref{eq:mono-sin}, the approximation of $\sin\pi x$ by $Q_m(x)$ is monotonic in $m$, whereas the approximation of $\sin\pi x$ by $S_m(x)$ is alternating: 
$$S_{2j}(x)<S_{2j+2}(x)<\sin\pi x<S_{2j-1}(x)<S_{2j+1}(x)$$ 
for all natural $j$ and all real $x>0$. 
These observations are illustrated in Fig.\ \ref{fig:Q,S}.

\begin{figure}[h]%
\includegraphics[width=\columnwidth]{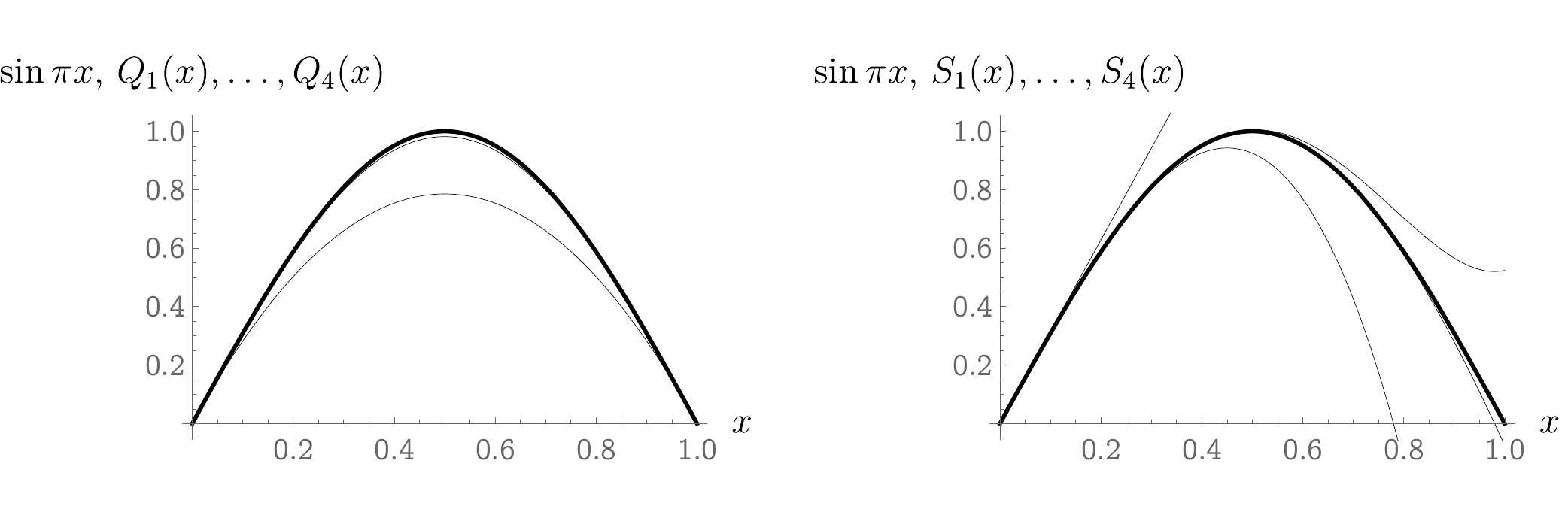}%
\caption{Left panel: Graphs $\{(x,\sin\pi x)\colon x\in[0,1]\}$ (thick) and $\{(x,Q_m(x))\colon x\in[0,1]\}$ (thin) for $m=1,2,3,4$. 
Right panel: Graphs $\{(x,\sin\pi x)\colon x\in[0,1]\}$ (thick) and $\{(x,S_m(x))\colon x\in[0,1]\}$ (thin) for $m=1,2,3,4$. 
}%
\label{fig:Q,S}%
\end{figure}
We see that $Q_3(x)$ and $Q_4(x)$ are visually indistinguishable from $\sin\pi x$ for $x\in[0,1]$; in contrast, $S_1(x),S_2(x),S_3(x),S_4(x)$ are all visually distinguishable from $\sin\pi x$ for $x\in[0,1]$. \qed 
\end{remark}

\begin{remark}\label{rem:example}
Inequalities \eqref{eq:mono} and \eqref{eq:mono-sin} can of course be used to prove other inequalities, which may have exactness or near-exactness properties. For example, we can quickly prove that 
\begin{equation*}
	f(x):=\frac{4}{9}+15 x^2-8 x+\frac{4 \left(2 \sin ^2(\pi  x)+\sin ^2(2 \pi  x)\right)}{\pi ^2}>0
\end{equation*}
for $x\in[0,1/2]$. Indeed, by \eqref{eq:mono-sin}, we have $f\ge f_4$ on $[0,1/2]$, where $f_4$ is the polynomial function obtained from $f$ by replacing the function $u\mapsto\sin\pi u$ in the above expression for $f$ by the polynomial function $Q_4$. The positivity of any polynomial on any interval can be verified purely algorithmically, which in this case gives $f_4>0$ on $(0,1/2]$, and hence $f>0$ on $[0,1/2]$. The graphs of the functions $f$ and $f-f_4$ are shown in Fig.\ \ref{fig:f,dif}. 
\begin{figure}[h]%
\includegraphics[width=\columnwidth]{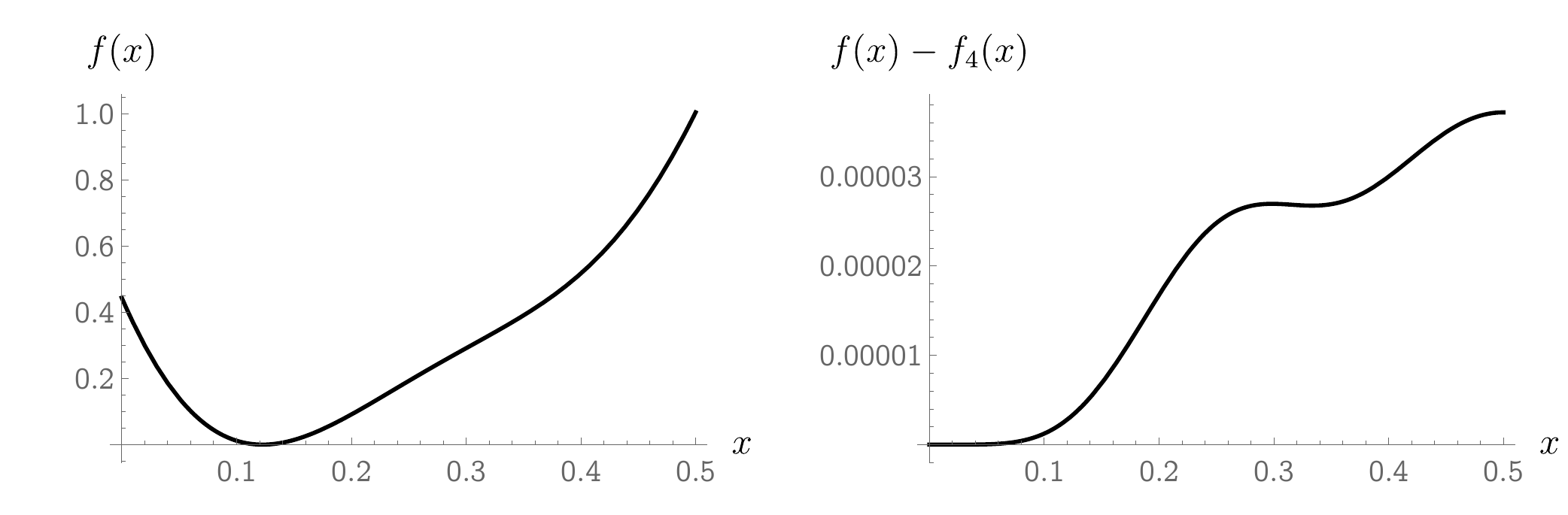}%
\caption{Graphs of 
of the functions $f$ (left panel) and $f-f_4$ (right panel).}%
\label{fig:f,dif}%
\end{figure}
\qed
\end{remark}

\section{Proofs}\label{proofs}

\begin{proof}[Proof of Theorem~\ref{th:}]
Take any real $x$ and let 
\begin{equation}\label{eq:y}
	y:=1/4-x^2,
\end{equation}
so that $y\le1/4$ and 
\begin{align}
	\cos\pi x=f(y):=\cos\Big(\frac\pi2\,\sqrt{1-4y}\Big)
	&=\sum_{n=0}^\infty\frac{(-1)^n}{(2n)!}\,\Big(\frac\pi2\,\sqrt{1-4y}\Big)^{2n} \notag \\ 
	&=\sum_{n=0}^\infty\frac{(-1)^n}{(2n)!}\,\Big(\frac{\pi^2}4\Big)^n
	(1-4y)^n \notag \\ 
	&=\sum_{n=0}^\infty\frac{(-1)^n}{(2n)!}\,\Big(\frac{\pi^2}4\Big)^n
	\sum_{j=0}^n\binom nj (-4y)^j \notag \\ 
	&=\sum_{j=0}^\infty(-4y)^j \sum_{n=j}^\infty \frac{(-1)^n}{(2n)!}\,\Big(\frac{\pi^2}4\Big)^n
	\binom nj  \label{eq:fubini} \\ 
	&=\sum_{j=0}^\infty(\pi^2 y)^j\, t_j=\sum_{j=1}^\infty(\pi^2 y)^j\, t_j; \label{eq:j=1}
\end{align}
the equality in \eqref{eq:fubini} follows by the Fubini theorem, 
the first equality in \eqref{eq:j=1} follows by the definition of $t_j$ in \eqref{eq:t,a}, 
and the second equality in \eqref{eq:j=1} follows because $t_0=f(0)=0$. 
Thus, identity \eqref{eq:expan} is proved. 

We have already noticed that $t_0=f(0)=0$. Similarly, $t_1=f'(0)/\pi^2=1/\pi$. As for \eqref{eq:recur}, it is the special case, with $z=\pi^2/4$, of the recurrence 
\begin{gather}\label{eq:T-recur}
	T_j(z)=\frac{2j-3}{2jz}\,T_{j-1}(z)-\frac1{4j(j-1)z}\,T_{j-2}(z)\quad\text{for }j=2,3,\dots, 
\end{gather}
where 
\begin{equation}\label{eq:T:=}
	T_j(z):=\sum_{k=0}^\infty 
\frac{(-z)^k}{(2j+2k)!}\,\binom{j+k}j,      
\end{equation}
so that $t_j=T_j(\pi^2/4)$. 
In turn, identity \eqref{eq:T-recur} can be verified by a straightforward comparison of the coefficients of the powers of $z$ on both sides of the identity. 

Next, by \eqref{eq:t,a}, for $j=1,2,\dots$ and $k=0,1,\dots$ the ratio
\begin{equation*}
	\frac{a_{j,k+1}}{-a_{j,k}}=\frac{\pi^2}{8(k+1)(2j+2k+1)} 
\end{equation*}
is positive and less than $1$, and this ratio tends to $0$ uniformly in $k=0,1,\dots$ as $j\to\infty$. Therefore, 
$0<t_j<a_{j,0}=\frac1{(2j)!}$
for all $j=1,2,\dots$, and 
$t_j\sim a_{j,0}=\frac1{(2j)!}\quad\text{as }j\to\infty$. 
which verifies the last sentence of Theorem~\ref{th:}. 
\end{proof}

\begin{proof}[Proof of Corollary~\ref{cor:}]
The inequalities in \eqref{eq:mono} follow immediately from \eqref{eq:P_m}, \eqref{eq:expan}, and the first inequality in \eqref{eq:<t_j<}. 

Recalling the definition of $\de_m(x)$ in \eqref{eq:bound}, identity \eqref{eq:expan}, the definition \eqref{eq:y} of $y$, and the second inequality in \eqref{eq:<t_j<}, 
we see that 
\begin{equation*}
\de_m(x)=\sum_{j=m+1}^\infty t_j \pi^{2j} y^j<\sum_{j=m+1}^\infty b_j(y)   
\end{equation*}
for all $x\in(-1/2,1/2)$, 
where 
\begin{equation*}
b_j(y)	:=\frac{(\pi^2 y)^j}{(2j)!}. 
\end{equation*}
Moreover, for any natural $m$, any natural $j\ge m+1$, and any $y\in(0,1/4]$, 
\begin{equation*}
	\frac{b_{j+1}(y)}{b_j(y)}
	=\frac{\pi^2 y}{(2j+2)(2j+1)}\le\frac{\pi^2/4}{(2m+4)(2m+3)}=q_m<1,   
\end{equation*}
and $q_m\to0$ as $m\to\infty$. 
Thus, we have verified \eqref{eq:bound}, which completes the proof of Corollary~\ref{cor:}. 
\end{proof}

\begin{proof}[Proof of Proposition~\ref{prop:}]
Identity \eqref{eq:t=} is a special case, with $z=\pi^2/4$, of the identity
\begin{equation}\label{eq:T=}
		T_j(z)=\frac{\sqrt\pi}{j!\,2^{j+1/2} }\, z^{1/4-j/2}\,J_{j-1/2}(\sqrt z)    
\end{equation} 
for real $z>0$, with $T_j(z)$ as defined in \eqref{eq:T:=}. In turn, to verify identity \eqref{eq:T=}, it is enough to compare the coefficients of the corresponding powers of $z$ in both sides of \eqref{eq:T=}, which is done with the help of the identity 
\begin{equation*}
	\Ga(n+1/2)=\frac{\sqrt\pi\,(2n)!}{4^n n!}
\end{equation*}
for $n=0,1,\dots$, which in turn is easy to check by induction. 
\end{proof}

{\bf Acknowledgment: } Thank you due to a referee for references \cite{milo86,mastro-milo}. 

\def\cprime{$'$} \def\polhk#1{\setbox0=\hbox{#1}{\ooalign{\hidewidth
  \lower1.5ex\hbox{`}\hidewidth\crcr\unhbox0}}}
  \def\polhk#1{\setbox0=\hbox{#1}{\ooalign{\hidewidth
  \lower1.5ex\hbox{`}\hidewidth\crcr\unhbox0}}}
  \def\polhk#1{\setbox0=\hbox{#1}{\ooalign{\hidewidth
  \lower1.5ex\hbox{`}\hidewidth\crcr\unhbox0}}} \def\cprime{$'$}
  \def\polhk#1{\setbox0=\hbox{#1}{\ooalign{\hidewidth
  \lower1.5ex\hbox{`}\hidewidth\crcr\unhbox0}}} \def\cprime{$'$}
  \def\polhk#1{\setbox0=\hbox{#1}{\ooalign{\hidewidth
  \lower1.5ex\hbox{`}\hidewidth\crcr\unhbox0}}} \def\cprime{$'$}
  \def\cprime{$'$}

\end{document}